\documentclass[reqno]{amsart}
\usepackage{amssymb,amsmath,amsthm,xcolor,enumerate,hyperref,cleveref}

\newtheorem{thrm}{Theorem}[section]
\newtheorem{cor}[thrm]{Corollary}
\newtheorem{lem}[thrm]{Lemma}

\theoremstyle{definition}
\newtheorem{defn}[thrm]{Definition}

\newtheorem{rem}[thrm]{Remark}

\crefrangeformat{equation}{#3(#1)#4--#5(#2)#6}

\crefname{thrm}{Theorem}{Theorems}
\crefname{lem}{Lemma}{Lemmas}
\crefname{cor}{Corollary}{Corollaries}
\crefname{prop}{Proposition}{Propositions}
\crefname{defn}{Definition}{Definitions}
\crefname{exm}{Example}{Examples}
\crefname{rem}{Remark}{Remarks}
\crefname{section}{Section}{Sections}
\crefname{equation}{\unskip}{\unskip}
\crefname{enumi}{\unskip}{\unskip}

\newcommand{\id}{\mathrm{id}}

\newcommand{\af}{\alpha}

\newcommand{\Dt}{\Delta}
\newcommand{\e}{\epsilon}
\newcommand{\s}{\sigma}

\newcommand{\A}{\mathcal{A}}

\newcommand{\ad}{\mathrm{ad}}

\newcommand{\m}{^{-1}}
\newcommand{\tl}{\tilde}

\newcommand{\sst}{\subseteq}

\newcommand{\hder}[1]{\operatorname{\mathrm{HDer}}{#1}}
\newcommand{\ihder}[1]{\operatorname{\mathrm{IHDer}}{#1}}

\newcommand{\hd}[1]{\{{#1}_n\}_{n=0}^\infty}

\newcommand{\NN}{\mathbb N}

\begin{document}
	
	\title[Higher Derivations of Finitary Incidence Algebras]{Higher Derivations\\ of Finitary Incidence Algebras}
	
	\author{Ivan Kaygorodov}
	\address{CMCC, Universidade Federal do ABC, Santo Andr\'e, Brazil}
	\email{kaygorodov.ivan@gmail.com}
	
	\author{Mykola Khrypchenko}
	\address{Departamento de Matem\'atica, Universidade Federal de Santa Catarina,  Campus Reitor Jo\~ao David Ferreira Lima, Florian\'opolis --- SC, CEP: 88040--900, Brazil}
	\email{nskhripchenko@gmail.com}
	
	\author{Feng Wei}
	\address{School of Mathematics and Statistics, Beijing Institute of Technology, Beijing 100081, P. R. China}
	\email{daoshuo@hotmail.com, daoshuowei@gmail.com}
	
	\begin{abstract}
		Let $P$ be a partially ordered set, $R$ a commutative unital ring and $FI(P,R)$ the finitary incidence algebra of $P$ over $R$. We prove that each $R$-linear higher derivation of $FI(P,R)$ decomposes into the product of an inner higher derivation of $FI(P,R)$ and the higher derivation of $FI(P,R)$ induced by a higher transitive map on the set of segments of $P$.
	\end{abstract}
	
	\subjclass[2010]{Primary 16S50, 16W25; Secondary 16G20, 06A11}
	
	\keywords{Finitary incidence algebra, higher derivation, inner higher derivation, higher transitive map}
	
	\thanks{This work is partially supported by the Training Program of International Exchange and Cooperation of the Beijing Institute of Technology
		and RFBR 17-01-00258. The authors would like to thank the International Affair Office of Beijing Institute of Technology for its kind consideration and warm help.}
	
	\maketitle
	
	\section*{Introduction}\label{intro}
	
	Let $(P,\le)$ be a preordered set and $R$ be a commutative unital ring. 
	Assume that $P$ is locally finite, i.e. for any $x\le y$ in $P$ there are only finitely many $z\in P$ such that $x\le z\le y$.
	The {\it incidence algebra}~\cite{SpiegelDonnell} $I(P,R)$ of $P$ over $R$ is the set of functions
	$$
	\{f: P\times P\to R\mid f(x,y)=0\ \text{if}\ x\not\le y\}
	$$
	with the natural structure of an $R$-module and multiplication given by the convolution
	\begin{align*}
	(fg)(x,y)&=\sum_{x\le z\le y}f(x,z)g(z,y)
	\end{align*}
	for all $f,g\in I(P,R)$ and $x,y\in P$. It would be helpful to point out that the full matrix algebra $M_n(R)$ as well as the upper triangular matrix algebra $T_n(R)$ are particular cases of incidence algebras. In addition, in the theory of operator algebras, the incidence algebra $I(P,R)$ of a finite poset $P$ is referred to as a bigraph algebra or a finite dimensional commutative subspace lattice algebra.

	Incidence algebras appeared in the early work by Ward~\cite{Ward} as generalized algebras of arithmetic functions. Later, they were extensively used as the fundamental tool of enumerative combinatorics in the series of works ``On the foundations of combinatorial theory''~\cite{Rota64,Mullin-Rota70,Goldman-Rota70,Doubilet-Rota-Stanley72} (see also the monograph~\cite{Stanley}). The study of algebraic mappings on incidence algebras was initiated by Stanley~\cite{Stanley70}. Since then, automorphisms, involutions, derivations (and their generalizations) on incidence algebras have been actively investigated, see \cite{Baclawski72, Spiegel93, Coelho-Polcino93, Coelho93, Coelho94, Koppinen95, Spiegel01, Kh-aut, Kh-der, Dugas12, Xiao15, Khr16, Zhang-Khrypchenko, Khr18-loc, BFK17, CDH18} and the references therein. 
	
	There are many interesting generalizations of derivations (see, for example, \cite{kp1,kp2} and their references), one of them being a higher derivation. Higher derivations are an active subject of research in (not necessarily associative or commutative) algebra. Firstly, higher
	derivations have close relationship with derivations. It should be remarked that the first component $d_1$ of each higher derivation $D=\{d_n\}_{n=0}^\infty$ of an algebra $\A$ is itself a derivation of $\A$. Conversely, let $d: \A\to\A$ be an ordinary derivation of an algebra $\A$ over a field of characteristic zero, then $D=\{\frac{1}{n!}d^n\}_{n=0}^\infty$ is a higher derivation of $\A$. Heerema~\cite{Heerema70}, Mirzavaziri~\cite{Mirzavaziri10} and Saymeh~\cite{Saymeh86} independently proved that each higher derivation of an algebra $\A$ over a field of characteristic zero is a combination of compositions of derivations, and hence one can characterize all higher derivations on $\A$ in terms of the derivations on $\A$. Ribenboim systemically studied higher derivations of arbitrary rings and those of arbitrary modules in~\cite{Rib71I, Rib71II}, where some familiar properties of derivations are generalized to the case of higher derivations. Ferrero and Haetinger found in~\cite{Ferr-Haet02b} the conditions under which Jordan higher derivations (or Jordan triple higher derivations) of a $2$-torsion-free (semi-)prime ring are higher derivations, and in~\cite{Ferr-Haet02a} the same authors studied higher derivations on (semi-)prime rings satisfying linear relations. Wei and Xiao \cite{Wei-Xiao11} described higher derivations of triangular algebras and related mappings, such as inner higher derivations, Jordan higher derivations, Jordan triple higher derivations and their generalizations.  
	
	The objective of this paper is to investigate higher derivations of finitary incidence algebras. Many researchers have made
	substantial contributions to the additive mapping theory of incidence algebras. Baclawski~\cite{Baclawski72} studied the automorphisms and derivations of incidence algebras $I(P,R)$ when $P$ is a locally finite partially ordered set. 
	In particular, he proved that every derivation of $I(P,R)$ can be decomposed as a sum of an inner derivation 
	and a derivation induced by a transitive map. Koppinen~\cite{Koppinen95} extended these results to the incidence algebras $I(P,R)$ with $P$ being a locally finite pre-ordered set. In~\cite{Xiao15}, Xiao characterized the derivations of $I(P,R)$ by a direct computation. Based on such a characterization of derivations, he proved that every Jordan derivation of $I(P,R)$ is a derivation provided that $R$ is $2$-torsion-free. Zhang and Khrypchenko~\cite{Zhang-Khrypchenko} considered Lie derivations of incidence algebras over $2$-torsion-free commutative unital rings. They proved that each Lie derivation $L$ of $I(P,R)$ can be represented as $L=D+F$, where $D$ is a derivation of $I(P,R)$ and $F$ is a linear mapping from $I(P,R)$ to its center. 
	
	More recently, special attention has been paid to additive mappings on finitary incidence algebras. Brusamarello, Fornaroli and Khrypchenko proved in~\cite{BFK17} that each $R$-linear Jordan isomorphism of the finitary incidence algebra $FI(P,R)$ of a partially ordered set $P$ over a $2$-torsion-free commutative unital ring $R$ onto an $R$-algebra $\A$ is the near-sum of a homomorphism and an anti-homomorphism. Brusamarello, Fornaroli and Santulo showed in~\cite{BFS12} that the finitary incidence algebra of an arbitrary poset $P$ over a field $K$ has an anti-automorphism (involution) if and only if $P$ has an anti-automorphism (involution). A decomposition theorem for such involutions was obtained in~\cite{BFS14}. Khrypchenko of the current article proved in~\cite{Khr18-loc} that each $R$-linear local derivation of the finitary incidence algebra $FI(P,R)$ of a poset $P$ over a commutative unital ring $R$ is a derivation, generalizing (partially) a result by Nowicki and Nowosad~\cite{Nowicki-Nowosad04}.
	
	The structure of this paper is as follows. In~\cref{prelim-sec} we collect some basic facts about higher derivations and finitary incidence algebras. These are used in~\cref{hider-FI-sec} to prove our main result \cref{d=Dt_rho*tl-sigma}.

	\section{Preliminaries}\label{prelim-sec}
	
	\subsection{Higher derivations}\label{hder-sec}
	Let $R$ be a ring. A sequence $d=\hd d$ of additive maps $R\to R$ is a \emph{higher derivation} of $R$ \emph{(of infinite order)}, if it satisfies
	\begin{enumerate}
		\item $d_0=\id_R$;
		\item $d_n(rs)=\sum_{i+j=n}d_i(r)d_j(s)$\label{d_n(rs)=sum-d_i(r)d_j(s)}
	\end{enumerate}
	for all $n\in\NN=\{1,2,\dots\}$ and $r,s\in R$. If \cref{d_n(rs)=sum-d_i(r)d_j(s)} holds for all $1\le n\le N$, then the sequence $\{d_n\}_{n=0}^N$ is called a \emph{higher derivation of order $N$}. Evidently, $\hd d$ is a higher derivation if and only if $\{d_n\}_{n=0}^N$ is a higher derivation of order $N$ for all $N\in\NN$. In particular, $d_1$ is always a usual derivation of $R$. 
	
	Denote by $\hder R$ the set of higher derivations of $R$ and consider the following operation on $\hder R$
	\begin{align}\label{d'-ast-d''-def}
	(d'*d'')_n=\sum_{i+j=n}d'_i\circ d''_j.
	\end{align}
	In particular, 
	\begin{align}\label{(d'-ast-d'')_1=d'_1+d''_1}
	(d'*d'')_1=d'_1+d''_1.
	\end{align}
	It was proved in~\cite{Heerema68} that $\hder R$ forms a group with respect to $*$, whose identity is the sequence $\hd \e$ with $\e_0=\id_R$ and 
	\begin{align}\label{e_i=zero}
	\e_n=0
	\end{align}
	for $n\in\NN$. 
	
	Given $r\in R$ and $k\in\NN$, define
	\begin{align}
	[r,k]_0&=\id_R,\label{[r_k]_0(x)-def}\\
	[r,k]_n(x)&=\begin{cases}
	0, & k\nmid n,\\
	r^lx-r^{l-1}xr, & n=kl,
	\end{cases}\label{[r_k]_n(x)-def}
	\end{align}
	for all $n\in\NN$ and $x\in R$. It was proved in \cite{Nowicki84-ider} that $[r,k]:=\hd{[r,k]}\in\hder R$, so that for any sequence $r=\{r_n\}_{n=1}^\infty\sst R$ one may define $\hd{(\Dt_r)}$ by means of $(\Dt_r)_0=\id_R$ and
	\begin{align}\label{(Dt_r)_n=([r_1_1]...[r_n_n])_n}
	(\Dt_r)_n=([r_1,1]*\dots *[r_n,n])_n,
	\end{align}
	where $n\in\NN$. Higher derivations of the form $\Dt_r$ will be called \emph{inner}. By \cite[Corollary 3.3]{Nowicki84-ider} the set of inner higher derivations forms a normal subgroup in $\hder R$, which will be denoted by $\ihder R$. In particular,
	\begin{align}\label{(Dt_r)_1-is-ad_(r_1)}
	(\Dt_r)_1(x)=[r_1,1]_1(x)=r_1x-xr_1
	\end{align} 
	is the usual inner derivation of $R$ associated with $r_1\in R$, which we denote by $\ad_{r_1}$.
	
	We shall begin with some formulas which were used in \cite{Nowicki84-der} without any proof.
	
	\begin{lem}\label{formulas-for-[r_1_1]...[r_n_n]_k}
		Let $n,k\in\NN$, such that $k<n$. Then for all $r\in R$
		\begin{align}\label{([r_n]^(-1))_k=0}
		([r,n]\m)_k=0.
		\end{align}
		Moreover, for any $\{r_n\}_{n=1}^\infty\sst R$
		\begin{enumerate}
			\item $([r_1,1]*\dots *[r_n,n])_k=([r_1,1]*\dots *[r_k,k])_k$;\label{[r_1_1]...[r_n_n]_k}
			\item $(([r_1,1]*\dots *[r_n,n])\m)_k=(([r_1,1]*\dots *[r_k,k])\m)_k$.\label{[r_1_1]...[r_n_n]^(-1)_k}
		\end{enumerate}
	\end{lem}
	\begin{proof}
		Since $[r,n]*[r,n]\m=\e$, it follows from \cref{d'-ast-d''-def,e_i=zero} that
		\begin{align}\label{sum-[r_n]_i-circ-[r_n]^(-1)_j}
		0=\e_k=([r,n]*[r,n]\m)_k=\sum_{i+j=k}[r,n]_i\circ([r,n]\m)_j.
		\end{align}
		But $i\le k<n$, so $[r,n]_i$ is non-zero only for $i=0$, in which case $[r,n]_i=\id_R$. Thus, the sum in \cref{sum-[r_n]_i-circ-[r_n]^(-1)_j} coincides with $([r,n]\m)_k$.
		
		\emph{Item \cref{[r_1_1]...[r_n_n]_k}.} By \cref{d'-ast-d''-def} one has
		\begin{align}\label{([r_1_1]...[r_n_n])_k=sum}
		([r_1,1]*\dots *[r_n,n])_k=\sum_{i_1+\dots+i_n=k}[r_1,1]_{i_1}\circ\dots \circ[r_n,n]_{i_n}.
		\end{align}
		Observe that $i_n\le k<n$ and hence $n\nmid i_n$, whenever $i_n\ne 0$. Then according to \cref{[r_k]_n(x)-def,[r_k]_0(x)-def} the map $[r_n,n]_{i_n}$ is non-zero only for $i_n=0$, and in this case $[r_n,n]_{i_n}=\id_R$. Consequently, the right-hand side of \cref{([r_1_1]...[r_n_n])_k=sum} equals
		\begin{align*}
		\sum_{i_1+\dots+i_{n-1}=k}[r_1,1]_{i_1}\circ\dots \circ[r_{n-1},{n-1}]_{i_{n-1}}=([r_1,1]*\dots *[r_{n-1},{n-1}])_k,
		\end{align*}
		and thus \cref{[r_1_1]...[r_n_n]_k} follows by applying the obvious induction argument.
		
		\emph{Item \cref{[r_1_1]...[r_n_n]^(-1)_k}.} As above, we see from \cref{d'-ast-d''-def} that
		\begin{align*}
		(([r_1,1]*\dots *[r_n,n])\m)_k&=([r_n,n]\m*\dots*[r_1,1]\m)_k\\
		&=\sum_{i_1+\dots+i_n=k}([r_n,n]\m)_{i_n}\circ\dots \circ([r_1,1]\m)_{i_1}.
		\end{align*}
		Now $([r_n,n]\m)_{i_n}=0$, except for the case $i_n=0$ in view of \cref{([r_n]^(-1))_k=0}. The rest of the proof follows as in \cref{[r_1_1]...[r_n_n]_k}.
	\end{proof}
	
	\begin{cor}\label{[r_n]^(-1)_n=-ad_r}
		For all $n\in\NN$ and $r\in R$
		\begin{align}\label{[r_n]^(-1)_n=-[r_n]_n=-ad_r}
		([r,n]\m)_n=-[r,n]_n=-\ad_r.
		\end{align}
	\end{cor}
	\begin{proof}
		Indeed,
		\begin{align*}
		0=\e_n=([r,n]*[r,n]\m)_n=\sum_{i+j=n}[r,n]_i\circ([r,n]\m)_j=[r,n]_n+([r,n]\m)_n
		\end{align*}
		by \cref{([r_n]^(-1))_k=0}, whence \cref{[r_n]^(-1)_n=-[r_n]_n=-ad_r} in view of \cref{[r_k]_n(x)-def}.
	\end{proof}
	
	\begin{lem}\label{d^(r)-properties}
		Let $d\in\hder R$, $r=\{r_n\}_{n=1}^\infty\sst R$ and $k\in\NN$. Define $d^{(r,k)}\in\hder R$ as being
		\begin{align}\label{d^(r_k)=([r_1_1]...[r_k_k])^(-1)d}
		d^{(r,k)}=([r_1,1]*\dots *[r_k,k])\m*d.
		\end{align}
		Moreover, set $d^{(r)}_0=\id_R$ and for all $n\in\NN$
		\begin{align}\label{d^(r)_n=d^(r_n)_n}
		d^{(r)}_n=d^{(r,n)}_n.
		\end{align}
		Then 
		\begin{enumerate}
			\item $d^{(r,k)}_l=d^{(r,l)}_l=d^{(r)}_l$ for all $1\le l\le k$;\label{d^(r_k)_l=d^(r_l)_l}
			\item $d^{(r)}\in\hder R$;\label{d^(r)-is-hder}
			\item $d=\Dt_r*d^{(r)}$.\label{d=Dt_r-ast-d^(r)}
		\end{enumerate}
	\end{lem}
	\begin{proof}
		For \cref{d^(r_k)_l=d^(r_l)_l} observe from \cref{formulas-for-[r_1_1]...[r_n_n]_k}\cref{[r_1_1]...[r_n_n]^(-1)_k} that 
		\begin{align*}
		d^{(r,k)}_l&=(([r_1,1]*\dots *[r_k,k])\m*d)_l\\
		&=\sum_{i+j=l}(([r_1,1]*\dots *[r_k,k])\m)_i\circ d_j\\
		&=\sum_{i+j=l}(([r_1,1]*\dots *[r_l,l])\m)_i\circ d_j\\
		&=(([r_1,1]*\dots *[r_l,l])\m*d)_l\\
		&=d^{(r,l)}_l.
		\end{align*}
		Then \cref{d^(r)-is-hder} automatically follows from the fact that for each fixed $N\in\NN$ the sequence $\{d^{(r)}_n\}_{n=0}^N$ coincides with the first $N$ terms of the sequence $d^{(r,N)}$.
		
		Now using \cref{formulas-for-[r_1_1]...[r_n_n]_k}\cref{[r_1_1]...[r_n_n]^(-1)_k}, \cref{d'-ast-d''-def,(Dt_r)_n=([r_1_1]...[r_n_n])_n,d^(r_k)=([r_1_1]...[r_k_k])^(-1)d,d^(r)_n=d^(r_n)_n,e_i=zero} we obtain for all $n\in\NN$ that
		\begin{align*}
		(\Dt_r*d^{(r)})_n&=\sum_{i+j=n}(\Dt_r)_i\circ d^{(r)}_j\\
		&=\sum_{i+j=n}([r_1,1]*\dots *[r_i,i])_i\circ (([r_1,1]*\dots *[r_j,j])\m*d)_j\\
		&=\sum_{i+j=n}([r_1,1]*\dots *[r_i,i])_i\circ \left(\sum_{k+l=j}([r_1,1]*\dots *[r_j,j])\m)_k\circ d_l\right)\\ 
		&=\sum_{i+j=n}([r_1,1]*\dots *[r_i,i])_i\circ \left(\sum_{k+l=j}([r_1,1]*\dots *[r_k,k])\m)_k\circ d_l\right)\\ 
		&=\sum_{i+j+k=n}([r_1,1]*\dots *[r_i,i])_i\circ (([r_1,1]*\dots *[r_j,j])\m)_j\circ d_k\\ 
		&=\sum_{k=0}^n\left(\sum_{i+j=n-k}([r_1,1]*\dots *[r_i,i])_i\circ (([r_1,1]*\dots *[r_j,j])\m)_j\right)\circ d_k\\		
		&=\sum_{k=0}^n\e_{n-k}\circ d_k\\
		&=d_n,
		\end{align*}
		and thus \cref{d=Dt_r-ast-d^(r)} holds.
	\end{proof}
	
	\subsection{Finitary incidence algebra}\label{FI-sec}
	Let $P$ be a poset and $R$ a commutative unital ring. Recall from \cite{Khripchenko-Novikov09} that a \emph{finitary series} is a formal sum of the form
	\begin{align}\label{formal-sum-in-I(P_R)}
	\alpha=\sum_{x\le y}\af_{xy}e_{xy},
	\end{align}
	where $x,y\in P$, $\af_{xy}\in R$ and $e_{xy}$ is a symbol, such that for any pair $x<y$ there exists only a finite number of $x\le u<v\le y$ with $\af_{uv}\ne 0$. The set of finitary series, denoted by $FI(P,R)$, possesses the natural structure of an $R$-module.  Moreover, it is closed under the convolution
	\begin{align}\label{conv-in-I(P_R)}
	\alpha\beta=\sum_{x\le y}\left(\sum_{x\le z\le y}\alpha_{xz}\beta_{zy}\right)e_{xy}.
	\end{align}
	Thus, $FI(P,R)$ is an algebra, called the \emph{finitary incidence algebra} of $P$ over $R$. The identity element of $FI(P,R)$ is the series $\delta=\sum_{x\in P}1_Re_{xx}$. Here and in what follows we adopt the next convention. If in~\cref{formal-sum-in-I(P_R)} the indices run through a subset $X$ of the ordered pairs $(x,y)$, $x,y\in P$, $x\le y$, then $\alpha_{xy}$ is meant to be zero for $(x,y)\not\in X$.
	
	Observe that
	\begin{align}\label{e_xy-cdot-e_uv}
	e_{xy}\cdot e_{uv}=
	\begin{cases}
	e_{xv}, & \mbox{if $y=u$},\\
	0, & \mbox{otherwise}.
	\end{cases}
	\end{align}
	In particular, the elements $e_x:=e_{xx}$, $x\in P$, are pairwise orthogonal idempotents of $FI(P,R)$, and for any $\af\in FI(P,R)$
	\begin{align}\label{e_x-alpha-e_y}
	e_x\af e_y=\begin{cases}
	\af_{xy}e_{xy}, & \mbox{ if }x\le y,\\
	0, & \mbox{ otherwise}.
	\end{cases}
	\end{align}
	
	Given $X\subseteq P$, we shall use the notation $e_X$ for the idempotent $\sum_{x\in X}1_R e_{xx}$. In particular, $e_x=e_{\{x\}}$. Note that $e_Xe_Y=e_{X\cap Y}$, so $e_xe_X=e_x$ for $x\in X$, and $e_xe_X=0$ otherwise.
	
	\section{Higher derivations of \texorpdfstring{$FI(P,R)$}{FI(P,R)}}\label{hider-FI-sec}
	
	\begin{lem}\label{hder-with-d(e_x)=0}
		Let $\{d_n\}_{n=0}^N$ be a higher derivation of $FI(P,R)$ of order $N\in\NN$, such that
		\begin{align}\label{d_e(e_x)-is-zero}
		d_n(e_x)=0
		\end{align}
		for all $x\in P$ and $1\le n<N$. Then for any $X\sst P$ and $x\in X$
		\begin{align}\label{d_N(e_x)-equals-e_xd_N(e_X)+d_N(e_x)e_X}
		d_N(e_x)=e_xd_N(e_X)+d_N(e_x)e_X.
		\end{align}
		In particular, for all $x<y$
		\begin{enumerate}
			\item $d_N(e_X)_{xy}=d_N(e_x)_{xy}$, if $x\in X$ and $y\not\in X$; \label{d_N(e_X)_xy-equas-d_N(e_x)_xy}
			\item $d_N(e_X)_{xy}=0_{xy}$, if $x,y\in X$. \label{d_N(e_X)_xy-equals-0_xy}
		\end{enumerate}
	\end{lem}
	\begin{proof}
		Since $e_x=e_x\cdot e_X$, we have
		\begin{align}
		d_N(e_x)&=\sum_{i+j=N}d_i(e_x)d_j(e_X)\notag\\
		&=e_xd_N(e_X)+d_N(e_x)e_X\notag\\
		&\quad+\sum_{i+j=N,\ i,j<N}d_i(e_x)d_j(e_X),\label{sum-d_i(e_x)d_j(e_X)}
		\end{align}
		the sum \cref{sum-d_i(e_x)d_j(e_X)} being zero by \cref{d_e(e_x)-is-zero}, whence \cref{d_N(e_x)-equals-e_xd_N(e_X)+d_N(e_x)e_X}. Now \cref{d_N(e_X)_xy-equas-d_N(e_x)_xy,d_N(e_X)_xy-equals-0_xy} follow by taking the coefficients of both sides of \cref{d_N(e_x)-equals-e_xd_N(e_X)+d_N(e_x)e_X} at $e_{xy}$.
	\end{proof}
	
	\begin{cor}\label{d_N(e_x)-properties}
		Let $\{d_n\}_{n=0}^N$ be a higher derivation of $FI(P,R)$ of order $N\in\NN$ satisfying \cref{d_e(e_x)-is-zero} for all $x\in P$ and $1\le n<N$. Then 
		\begin{align}\label{d_N(e_x)_xy=-d_N(e_x)_xy}
		d_N(e_x)_{xy}=-d_N(e_y)_{xy}.
		\end{align}
	\end{cor}
	\begin{proof}
		Indeed, \cref{d_N(e_x)_xy=-d_N(e_x)_xy} follows from \cref{hder-with-d(e_x)=0}\cref{d_N(e_X)_xy-equals-0_xy} with $X=\{x,y\}$ and the easy observation that $e_{\{x,y\}}=e_x+e_y$.
	\end{proof}

	\begin{lem}\label{decomp-Dt_bt*d'}
		Let $d=\hd d\in\hder{FI(P,R)}$. Then there is $\rho=\{\rho_n\}_{n=1}^\infty\sst FI(P,R)$ such that for all $n\in\NN$ and $x\in P$
		\begin{align}\label{d^(rho)_n(e_x)=zero}
		d^{(\rho)}_n(e_x)=0,
		\end{align}
		where $d^{(\rho)}$ is given by \cref{d^(r_k)=([r_1_1]...[r_k_k])^(-1)d,d^(r)_n=d^(r_n)_n}.
	\end{lem}
	\begin{proof}
		Define
		\begin{align}
		(\rho_1)_{xy}&=d_1(e_y)_{xy},\label{(rho_1)_xy=d_1(e_y)_xy}\\
		(\rho_n)_{xy}&=(([\rho_1,1]*\dots *[\rho_{n-1},n-1])\m*d)_n(e_y)_{xy},\ n\in\NN,\ n>1.\label{(rho_n)_xy=(([rho_1_1]*... *[rho_n-1_n-1])^(-1)*d)_n(e_y)_xy}
		\end{align}
		We shall prove that $\rho_n\in FI(P,R)$ and \cref{d^(rho)_n(e_x)=zero} holds by induction on $n$. 
		
		Since $d_1$ is a usual derivation of $FI(P,R)$ and 
		\begin{align*}
		d_1=(\Dt_\rho)_1+d^{(\rho)}_1=\ad_{\rho_1}+d^{(\rho)}_1
		\end{align*}
		by \cref{d^(r)-properties}\cref{d=Dt_r-ast-d^(r)}, \cref{(d'-ast-d'')_1=d'_1+d''_1,(Dt_r)_1-is-ad_(r_1)}, the case $n=1$ is exactly \cite[Lemma 2]{Kh-der} (compare \cref{(rho_1)_xy=d_1(e_y)_xy} with formula (7) from \cite{Kh-der}). 
		
		Now assume that $\rho_n\in FI(P,R)$ and \cref{d^(rho)_n(e_x)=zero} is true for all $0<n<m$ and $x\in P$. In particular, $d^{(\rho,n)}$ is a well-defined higher derivation of $FI(P,R)$ for each $0<n<m$ (see \cref{d^(r_k)=([r_1_1]...[r_k_k])^(-1)d}). 
		
		We first show that $\rho_m\in FI(P,R)$. Suppose that there are $x<y$ and an infinite set $S$ of pairs $(u,v)$, such that $x\le u<v\le y$ and $(\rho_m)_{uv}\ne 0$. Observe from \cref{(rho_n)_xy=(([rho_1_1]*... *[rho_n-1_n-1])^(-1)*d)_n(e_y)_xy} that $(\rho_m)_{uv}=d^{(\rho,m-1)}_m(e_v)_{uv}$. Since $d^{(\rho,m-1)}_m(e_v)$ is a finitary series, for each $v$ there is only a finite number of $u$, such that $(u,v)\in S$. Moreover, $d^{(\rho,m-1)}_n=d^{(\rho)}_n$ for all $n<m$ in view of \cref{d^(r)-properties}\cref{d^(r_k)_l=d^(r_l)_l}. Consequently, $d^{(\rho,m-1)}_n(e_x)=0$ for all $0<n<m$ and $x\in P$ by the induction hypothesis, and thus we may apply \cref{d_N(e_x)-properties,hder-with-d(e_x)=0} to $\{d^{(\rho,m-1)}_n\}_{n=0}^m$.  We have by \cref{d_N(e_x)_xy=-d_N(e_x)_xy}
		\begin{align}\label{d^(rho_m-1)_m(e_u)_uv-ne-0_uv}
		d^{(\rho,m-1)}_m(e_u)_{uv}=-d^{(\rho,m-1)}_m(e_v)_{uv}=-(\rho_m)_{uv}\ne 0.
		\end{align}
		Since $d^{(\rho,m-1)}_m(e_u)\in FI(P,R)$, it follows that for each $u$ there is only a finite number of $v$, such that $(u,v)\in S$. Therefore, as in \cite[Lemma 2]{Kh-der} we may construct an infinite $S'\sst S$, such that the sets 
		\begin{align*}
		U=\{u\mid (u,v)\in S'\}\mbox{ and }V=\{v\mid (u,v)\in S'\}
		\end{align*}
		are infinite and disjoint. But then $d^{(\rho,m-1)}_m(e_U)_{uv}=d^{(\rho,m-1)}_m(e_u)_{uv}\ne 0$ in view of \cref{hder-with-d(e_x)=0}\cref{d_N(e_X)_xy-equas-d_N(e_x)_xy} and \cref{d^(rho_m-1)_m(e_u)_uv-ne-0_uv}, contradicting the fact that $d^{(\rho,m-1)}_m(e_U)\in FI(P,R)$.
		
		Now, under the same hypothesis assumption as above, we prove \cref{d^(rho)_n(e_x)=zero}. We have already shown that $\rho_m\in FI(P,R)$. So, using \cref{([r_n]^(-1))_k=0,[r_n]^(-1)_n=-ad_r} we have
		\begin{align}
		d^{(\rho)}_m&=([\rho_m,m]\m*\dots*[\rho_1,1]\m*d)_m\notag\\
		&=\sum_{i+j=m}([\rho_m,m]\m)_i\circ([\rho_{m-1},m-1]\m*\dots*[\rho_1,1]\m*d)_j\notag\\
		&=([\rho_m,m]\m)_m+([\rho_{m-1},m-1]\m*\dots*[\rho_1,1]\m*d)_m\notag\\
		&=-\ad_{\rho_m}+d^{(\rho,m-1)}_m.\label{d^(rho)_m=-ad_rho_m+d^(rho_m-1)_m}
		\end{align}
		Notice that
		\begin{align}
		(d^{(\rho,m-1)}_m(e_x)e_x)_{uv}&=
		\begin{cases}
		(d^{(\rho,m-1)}_m(e_x))_{ux}, & v=x,\\
		0, & v\ne x,
		\end{cases}\notag\\
		&=
		\begin{cases}
		(\rho_m)_{ux}, & v=x,\\
		0, & v\ne x,
		\end{cases}\notag\\
		&=(\rho_me_x)_{uv}.\label{d^{(rho_m-1)}_m(e_x)e_x=rho_me_x}
		\end{align}
		Moreover, since
		\begin{align*}
		d^{(\rho,m-1)}_m(e_x)_{xy}=-d^{(\rho,m-1)}_m(e_y)_{xy}=-(\rho_m)_{xy}
		\end{align*}
		by \cref{d_N(e_x)_xy=-d_N(e_x)_xy}, we similarly get that
		\begin{align}\label{e_xd^(rho_m-1)_m(e_x)=-e_x-rho_m}
		e_xd^{(\rho,m-1)}_m(e_x)=-e_x\rho_m.
		\end{align}
		Thus, in view of \cref{d_N(e_x)-equals-e_xd_N(e_X)+d_N(e_x)e_X}, \cref{d^{(rho_m-1)}_m(e_x)e_x=rho_me_x,e_xd^(rho_m-1)_m(e_x)=-e_x-rho_m}
		\begin{align*}
		d^{(\rho,m-1)}_m(e_x)=e_xd^{(\rho,m-1)}_m(e_x)+d^{(\rho,m-1)}_m(e_x)e_x=\rho_me_x-e_x\rho_m=\ad_{\rho_m}(e_x).
		\end{align*}
		Combining this with \cref{d^(rho)_m=-ad_rho_m+d^(rho_m-1)_m} we get $d^{(\rho)}_m(e_x)=0$, which completes the induction step and thus proves \cref{d^(rho)_n(e_x)=zero}.
	\end{proof}
	
	Thus, it suffices to describe the higher derivations $d$ of $FI(P,R)$ whose terms annihilate $e_x$ for all $x\in P$. We shall give an equivalent characterization of such $d$, assuming that all $d_n$ are $R$-linear.
	
	The following definition is due to Nowicky~\cite{Nowicki84-der}. 
	\begin{defn}\label{trans-map-defn}
		A sequence $\s=\hd{\s}$ of maps on $I=\{(x,y)\in P\times P\mid x\le y\}$ with values in $R$ is called a \emph{higher transitive map}, if
		\begin{enumerate}
			\item $\s_0(x,y)=1_R$ for all $x\le y$;\label{s_0(x_y)=1-for-all-x_y}
			\item $\s_n(x,y)=\sum_{i+j=n}\s_i(x,z)\s_j(z,y)$ for all $x\le z\le y$.\label{s_n(x_y)=sum-s_i(x_z)s_j(z_y)}
		\end{enumerate}
	\end{defn}
	
	\begin{rem}\label{s_n(x_x)=0-for-all-x}
		If $\s$ is a higher transitive map, then
		\begin{align}\label{s_n(xx)-is-zero}
		\s_n(x,x)=0
		\end{align}
		for all $n\in\NN$ and $x\in P$.
	\end{rem}
	\begin{proof}
		Indeed, $\s_1(x,x)=\s_1(x,x)+\s_1(x,x)$, so $\s_1(x,x)=0$. Now suppose that the equality holds for all $n<m$. Then
		\begin{align*}
		\s_m(x,x)&=\sum_{i+j=m}\s_i(x,x)\s_j(x,x)\\
		&=1_R\cdot\s_m(x,x)+\s_m(x,x)\cdot 1_R\\
		&\quad+\sum_{i+j=m,\ i,j<m}\s_i(x,x)\s_j(x,x)\\
		&=2\s_m(x,x)
		\end{align*}
		by the induction hypothesis. Thus, $\s_m(x,x)=0$.
	\end{proof}
	
	\begin{lem}\label{tl-sigma-defn}
		Given a higher transitive map $\s$, denote by $\tl\s$ the following sequence of maps $FI(P,R)\to FI(P,R)$
		\begin{align*}
		\tl\s_n(\af)=\sum_{x\le y}\s_n(x,y)\af_{xy}e_{xy},
		\end{align*}
		where $n\in\NN\cup\{0\}$ and $\af\in FI(P,R)$. Then $\tl\s\in\hder{FI(P,R)}$.
	\end{lem}
	\begin{proof}
		It is obvious that $\tl\s_n$ is well-defined and additive. The fact that $\tl\s$ satisfies \cref{d_n(rs)=sum-d_i(r)d_j(s)} of the definition of a higher derivation is easy to verify (see, for example the proof of \cite[Lemma 3.6]{Nowicki84-der}).
	\end{proof}
	
	\begin{lem}\label{d(e_x)=0<=>s=tl-sigma}
		Let $d=\hd d\in\hder{FI(P,R)}$ be $R$-linear. Then 
		\begin{align}\label{d_n(e_x)-s-zero}
		d_n(e_x)=0
		\end{align}
		for all $n\in\NN$ and $x\in P$ if and only if $d=\tl\s$ for some transitive map $\s$.
	\end{lem}
	\begin{proof}
		Clearly, $d=\tl\s$ is $R$-linear and satisfies \cref{d_n(e_x)-s-zero} in view of \cref{s_n(xx)-is-zero}.
		
		Let us prove the converse. Assume \cref{d_n(e_x)-s-zero} and define
		\begin{align}\label{s_n(x_y)=d_n(e_xy)_xy}
		\s_n(x,y)=d_n(e_{xy})_{xy}.
		\end{align}
		Observe from \cref{d_n(rs)=sum-d_i(r)d_j(s)} of the definition of a higher derivation, \cref{e_x-alpha-e_y,d_n(e_x)-s-zero} that, given $\af\in FI(P,R)$ and $x\le y$,
		\begin{align*}
		d_n(\af_{xy}e_{xy})=d_n(e_x\af e_y)=\sum_{i+j+k=n}d_i(e_x)d_j(\af)d_k(e_y)=e_xd_n(\af)e_y=d_n(\af)_{xy}e_{xy}.
		\end{align*}
		Hence, using $R$-linearity, we conclude that
		\begin{align}\label{d_n(af)_xy=s_n(x_y)af_xy}
		d_n(\af)_{xy}=d_n(\af_{xy}e_{xy})_{xy}=\af_{xy}d_n(e_{xy})_{xy}=\s_n(x,y)\af_{xy},
		\end{align}
		so $d=\tl\s$. It remains to verify \cref{s_0(x_y)=1-for-all-x_y,s_n(x_y)=sum-s_i(x_z)s_j(z_y)} of \cref{trans-map-defn}. Condition \cref{s_0(x_y)=1-for-all-x_y} is simply the statement that $(e_{xy})_{xy}=1_R$ by \cref{s_n(x_y)=d_n(e_xy)_xy}. Now take $x\le z\le y$. Then, $e_{xy}=e_{xz}e_{zy}$ in view of \cref{e_xy-cdot-e_uv}, so that by \cref{conv-in-I(P_R),s_n(x_y)=d_n(e_xy)_xy,d_n(af)_xy=s_n(x_y)af_xy} and \cref{d_n(rs)=sum-d_i(r)d_j(s)} of the definition of a higher derivation
		\begin{align*}
		\s_n(x,y)=d_n(e_{xy})_{xy}&=\sum_{i+j=n}(d_i(e_{xz})d_j(e_{zy}))_{xy}\\
		&=\sum_{i+j=n}\sum_{x\le u\le y}d_i(e_{xz})_{xu}d_j(e_{zy})_{uy}\\
		&=\sum_{i+j=n}\sum_{x\le u\le y}\s_i(x,u)(e_{xz})_{xu}\s_j(u,y)(e_{zy})_{uy}\\
		&=\sum_{i+j=n}\s_i(x,z)\s_j(z,y),
		\end{align*}
		proving \cref{trans-map-defn}\cref{s_n(x_y)=sum-s_i(x_z)s_j(z_y)}.
	\end{proof}
	
	\begin{thrm}\label{d=Dt_rho*tl-sigma}
		Each $R$-linear higher derivation of $FI(P,R)$ is of the form $\Dt_\rho*\tl\s$ for some $\rho=\{\rho_n\}_{n=1}^\infty\sst FI(P,R)$ and a higher transitive map $\s$.
	\end{thrm}
	\begin{proof}
		This follows from \cref{d^(r)-is-hder,d=Dt_r-ast-d^(r)} of \cref{d^(r)-properties} and \cref{d(e_x)=0<=>s=tl-sigma,decomp-Dt_bt*d'}.
	\end{proof}
	
	\section{Acknowledgements}
	
	The authors are grateful to the reviewer whose suggestions helped them to improve the readability of the paper. 
	
	\bibliography{bibl}{}

\begin{thebibliography}{10}

\bibitem{Baclawski72}
{\sc Baclawski, K.}
\newblock {Automorphisms and derivations of incidence algebras}.
\newblock {\em Proc. Amer. Math. Soc. 36}, 2 (1972), 351--356.

\bibitem{BFK17}
{\sc Brusamarello, R., Fornaroli, {\'E}.~Z., and Khrypchenko, M.}
\newblock Jordan isomorphisms of finitary incidence algebras.
\newblock {\em Linear Multilinear Algebra 66}, 3 (2018), 565--579.

\bibitem{BFS12}
{\sc Brusamarello, R., Fornaroli, {\'E}.~Z., and Santulo, E.~A.}
\newblock {Anti-automorphisms and involutions on (finitary) incidence
  algebras}.
\newblock {\em Linear Multilinear Algebra 60}, 2 (2012), 181--188.

\bibitem{BFS14}
{\sc Brusamarello, R., Fornaroli, {\'E}.~Z., and Santulo, E.~A.}
\newblock {Classification of involutions on finitary incidence algebras}.
\newblock {\em Int. J. Algebra Comput. 24}, 8 (2014), 1085--1098.

\bibitem{Coelho93}
{\sc Coelho, S.~P.}
\newblock {The automorphism group of a structural matrix algebra}.
\newblock {\em Linear Algebra Appl. 195\/} (1993), 35--58.

\bibitem{Coelho94}
{\sc Coelho, S.~P.}
\newblock {Automorphism groups of certain structural matrix rings}.
\newblock {\em Comm. Algebra 22}, 14 (1994), 5567--5586.

\bibitem{Coelho-Polcino93}
{\sc Coelho, S.~P., and {Polcino Milies}, C.}
\newblock {Derivations of upper triangular matrix rings}.
\newblock {\em Linear Algebra Appl. 187\/} (1993), 263--267.

\bibitem{CDH18}
{\sc Courtemanche, J., Dugas, M., and Herden, D.}
\newblock Local automorphisms of finitary incidence algebras.
\newblock {\em Linear Algebra Appl. 541\/} (2018), 221--257.

\bibitem{Doubilet-Rota-Stanley72}
{\sc Doubilet, P., Rota, G.-C., and Stanley, R.~P.}
\newblock {On the foundations of combinatorial theory. VI. The idea of
  generating function}.
\newblock In {\em {Proceedings of the Sixth Berkeley Symposium on Mathematical
  Statistics and Probability}}, vol.~II: Probability theory. Univ. California
  Press, 1972, pp.~267--318.

\bibitem{Dugas12}
{\sc Dugas, M.}
\newblock {Homomorphisms of finitary incidence algebras}.
\newblock {\em Comm. Algebra 40}, 7 (2012), 2373--2384.

\bibitem{Ferr-Haet02b}
{\sc Ferrero, M., and Haetinger, C.}
\newblock {Higher derivations and a theorem by {H}erstein}.
\newblock {\em Quaest. Math. 25}, 2 (2002), 249--257.

\bibitem{Ferr-Haet02a}
{\sc Ferrero, M., and Haetinger, C.}
\newblock {Higher derivations of semiprime rings}.
\newblock {\em Comm. Algebra 30}, 5 (2002), 2321--2333.

\bibitem{Heerema68}
{\sc Heerema, N.}
\newblock Convergent higher derivations on local rings.
\newblock {\em Trans. Amer. Math. Soc. 132\/} (1968), 31--44.

\bibitem{Heerema70}
{\sc Heerema, N.}
\newblock Higher derivations and automorphisms of complete local rings.
\newblock {\em Bull. Amer. Math. Soc. 76\/} (1970), 1212--1225.

\bibitem{kp2}
{\sc Kaygorodov, I., and Popov, Y.}
\newblock {A characterization of non-associative nilpotent algebras by
  invertible Leibniz-derivations}.
\newblock {\em Journal of Algebra 456\/} (2016), 1086--1106.

\bibitem{kp1}
{\sc Kaygorodov, I., and Popov, Y.}
\newblock {Generalized derivations of (color) $n$-ary algebras}.
\newblock {\em Linear Multilinear Algebra 64}, 6 (2016), 1086--1106.

\bibitem{Kh-aut}
{\sc Khripchenko, N.~S.}
\newblock {Automorphisms of finitary incidence rings}.
\newblock {\em Algebra and Discrete Math. 9}, 2 (2010), 78--97.

\bibitem{Kh-der}
{\sc Khripchenko, N.~S.}
\newblock {Derivations of finitary incidence rings}.
\newblock {\em Comm. Algebra 40}, 7 (2012), 2503--2522.

\bibitem{Khripchenko-Novikov09}
{\sc Khripchenko, N.~S., and Novikov, B.~V.}
\newblock {Finitary incidence algebras}.
\newblock {\em Comm. Algebra 37}, 5 (2009), 1670--1676.

\bibitem{Khr16}
{\sc Khrypchenko, M.}
\newblock {Jordan derivations of finitary incidence rings}.
\newblock {\em Linear Multilinear Algebra 64}, 10 (2016), 2104--2118.

\bibitem{Khr18-loc}
{\sc Khrypchenko, M.}
\newblock Local derivations of finitary incidence algebras.
\newblock {\em Acta Math. Hungar. 154}, 1 (2018), 48--55.

\bibitem{Koppinen95}
{\sc Koppinen, M.}
\newblock {Automorphisms and Higher Derivations of Incidence Algebras}.
\newblock {\em J. Algebra 174\/} (1995), 698--723.

\bibitem{Mirzavaziri10}
{\sc Mirzavaziri, M.}
\newblock Characterization of higher derivations on algebras.
\newblock {\em Comm. Algebra 38}, 3 (2010), 981--987.

\bibitem{Nowicki84-der}
{\sc Nowicki, A.}
\newblock {Higher $R$-derivations of special subrings of matrix rings}.
\newblock {\em Tsukuba J. Math. 8}, 2 (1984), 227--253.

\bibitem{Nowicki84-ider}
{\sc Nowicki, A.}
\newblock {Inner derivations of higher orders}.
\newblock {\em Tsukuba J. Math. 8}, 2 (1984), 219--225.

\bibitem{Nowicki-Nowosad04}
{\sc Nowicki, A., and Nowosad, I.}
\newblock {Local derivations of subrings of matrix rings}.
\newblock {\em Acta Math. Hungar. 105}, 1--2 (2004), 145--150.

\bibitem{Rib71I}
{\sc Ribenboim, P.}
\newblock {Higher derivations of rings. {I}}.
\newblock {\em Rev. Roumaine Math. Pures Appl. 16\/} (1971), 77--110.

\bibitem{Rib71II}
{\sc Ribenboim, P.}
\newblock {Higher derivations of rings. {II}}.
\newblock {\em Rev. Roumaine Math. Pures Appl. 16\/} (1971), 245--272.

\bibitem{Rota64}
{\sc Rota, G.-C.}
\newblock {On the foundations of combinatorial theory. I. Theory of M{\"o}bius
  functions}.
\newblock {\em Z. Wahrscheinlichkeitstheorie und Verw. Gebiete 2}, 4 (1964),
  340--368.

\bibitem{Goldman-Rota70}
{\sc Rota, G.-C., and Goldman, J.}
\newblock {On the foundations of combinatorial theory. IV. Finite vector spaces
  and Eulerian generating functions}.
\newblock {\em Stud. In Appl. Math. 49\/} (1970), 239--258.

\bibitem{Mullin-Rota70}
{\sc Rota, G.-C., and Mullin, R.}
\newblock {On the foundations of combinatorial theory. III. Theory of binomial
  enumeration}.
\newblock In {\em {Graph Theory and its Appl.}}, B.~Harris, Ed. Acad. Press.,
  1970, pp.~167--213.

\bibitem{Saymeh86}
{\sc Saymeh, S.~A.}
\newblock On {H}asse-{S}chmidt higher derivations.
\newblock {\em Osaka J. Math. 23}, 2 (1986), 503--508.

\bibitem{Spiegel93}
{\sc Spiegel, E.}
\newblock {Automorphisms of incidence algebras}.
\newblock {\em Comm. Algebra 21}, 8 (1993), 2973--2981.

\bibitem{Spiegel01}
{\sc Spiegel, E.}
\newblock {On the automorphisms of incidence algebras}.
\newblock {\em J. Algebra 239\/} (2001), 615--623.

\bibitem{SpiegelDonnell}
{\sc {Spiegel}, E., and {O'Donnell}, C.~J.}
\newblock {\em {Incidence algebras}}.
\newblock New York, NY: Marcel Dekker, 1997.

\bibitem{Stanley}
{\sc Stanley, R.}
\newblock {\em {Enumerative Combinatorics}}, vol.~1.
\newblock Cambridge University Press, 1997.

\bibitem{Stanley70}
{\sc Stanley, R.~P.}
\newblock {Structure of incidence algebras and their automorphism groups}.
\newblock {\em Bull. Amer. Math. Soc. 76\/} (1970), 1236--1239.

\bibitem{Ward}
{\sc {Ward}, M.}
\newblock {Arithmetic functions on rings}.
\newblock {\em {Ann. Math. (2)} 38\/} (1937), 725--732.

\bibitem{Wei-Xiao11}
{\sc Wei, F., and Xiao, Z.}
\newblock {Higher derivations of triangular algebras and its generalizations}.
\newblock {\em Linear Algebra Appl. 435}, 5 (2011), 1034--1054.

\bibitem{Xiao15}
{\sc Xiao, Z.}
\newblock {Jordan derivations of incidence algebras}.
\newblock {\em Rocky Mountain J. Math. 45}, 4 (2015), 1357--1368.

\bibitem{Zhang-Khrypchenko}
{\sc Zhang, X., and Khrypchenko, M.}
\newblock {Lie derivations of incidence algebras}.
\newblock {\em Linear Algebra Appl. 513\/} (2017), 69--83.

\end{thebibliography}
	\bibliographystyle{acm}
	
\end{document}